\documentclass[11pt]{amsart}
\usepackage[utf8]{inputenc}
\usepackage{amsmath, amssymb, mathtools, palatino, euler, epic, eepic,floatflt, microtype}
\usepackage{graphicx}
\usepackage[usenames]{xcolor}
\usepackage{float}
\usepackage{lipsum}
\usepackage{fancyhdr, enumitem, tikz}

\usepackage[normalem]{ulem}
\usepackage{mathrsfs}
\usepackage{tikz}
\usepackage{systeme}

\usepackage[margin=1 in]{geometry}
\newcommand\defi[1]{\textit{\color{blue}#1}}    

\definecolor{darkblue}{rgb}{0,0,0.7}
\definecolor{darkred}{rgb}{0.7,0,0}
\usepackage[colorlinks,linkcolor=darkred, citecolor=darkblue,
pagebackref=true,pdftex]{hyperref}

\newtheorem{thm}{Theorem}[section]
\newtheorem{prop}[thm]{Proposition}
\newtheorem{lemma}[thm]{Lemma}
\newtheorem{conj}[thm]{Conjecture}
\newtheorem{cor}[thm]{Corollary}

\newtheorem{defn}[thm]{Definition}
\newtheorem{rem}[thm]{Remark}
\newtheorem{example}[thm]{Example}

\newtheorem{lem}[thm]{Lemma}

\newcommand\lk{\operatorname{\text{$\ell$k}}}   
\newcommand\lkc{\operatorname{\lk_{\Delta}}}
\newcommand\stc{\operatorname{\text{star}_{\Delta}}}
\newcommand\delc{\operatorname{\text{del}_{\Delta}}}
\newcommand\del{\operatorname{\text{del}}}
\newcommand\Cl{\operatorname{\text{Cl}}}

\newcommand\lkcp{\operatorname{\lk_{\Delta^\prime}}}

\newcommand\delcp{\operatorname{\text{del}_{\Delta^\prime}}} 

\newcommand{\M}{\mathcal{M}}

\title[Completing and extending shellings]{Completing and extending shellings \\
of vertex decomposable complexes}

\author{Michaela Coleman}
\author{Anton Dochtermann}
\author{Nathan Geist}
\author{Suho Oh}

\address{Department of Mathematics, University of Georgia, Athens, GA}
\email{mc08274@uga.edu}

\address{Department of Mathematics, Texas State University}
\email{dochtermann@txstate.edu}

\address{Department of Mathematics, Duke University, Durham, NC}
\email{nathan.geist@duke.edu}

\address{Department of Mathematics, Texas State University}
\email{s\_o79@txstate.edu}

\date{\today}

\begin{document}

\begin{abstract}
We say that a pure $d$-dimensional simplicial complex $\Delta$ on $n$ vertices is \emph{shelling completable} if $\Delta$ can be realized as the initial sequence of some shelling of $\Delta_{n-1}^{(d)}$, the $d$-skeleton of the $(n-1)$-dimensional simplex. A well-known conjecture of Simon posits that any shellable complex is shelling completable. 

In this note we prove that vertex decomposable complexes are shelling completable.  In fact we show that if $\Delta$ is a vertex decomposable complex then there exists an ordering of its ground set $V$ such that adding the revlex smallest missing $(d+1)$-subset of $V$ results in a complex that is again vertex decomposable. We explore applications to matroids and shifted complexes, as well as connections to ridge-chordal complexes and $k$-decomposability.  We also show that if $\Delta$ is a $d$-dimensional complex on at most $d+3$ vertices then the notions of shellable, vertex decomposable, shelling completable, and extendably shellable are all equivalent. 
\end{abstract}

\maketitle




\section{Introduction}\label{sec:intro}
A pure simplicial complex $\Delta$ is \emph{$k$-decomposable} if it is a simplex, or it admits a face of dimension at most $k$ whose deletion and link results in a complex that is again in the class. The case $k=0$ is also known as \emph{vertex decomposable} (see Definition \ref{def:vd}) and is an important class in many combinatorial settings. Examples of vertex decomposable complexes include the independence complexes of matroids (as shown by Provan and Billera in \cite{ProBil}, where the class of $k$-decomposable was introduced). It is also not hard to see that for any $k = 1, 2, \dots, n-1$ the $k$-skeleton of a simplex on vertex set $[n]$, which we denote $\Delta_{n-1}^{(k)}$, is vertex decomposable.

By definition a $k$-decomposable complex is $(k+1)$-decomposable, and a pure $d$-dimensional complex is $d$-decomposable complex if and only if it is \emph{shellable}.  Shellability of $\Delta$ can also be characterized in terms of the existence of an ordering of its facets $F_1, F_2, \dots, F_s$ that satisfies a codimension one intersection property (see Definition \ref{defn:shellable}). Shellability itself is an important combinatorial tool that has consequences for the topology of $\Delta$ as well as algebraic properties of its Stanley-Reisner ring.  By results of Bruggesser and Mani \cite{BruMan}, a large class of shellable simplicial complexes come from the boundary complexes of simplicial polytopes.


Given a pure $k$-decomposable complex $\Delta$, a natural question to ask is whether one can add a new facet to $\Delta$ that results in a complex that is still $k$-decomposable. In the case of shellable complexes this question was asked by Simon in \cite{Sim}, in the context of the related notion of extendable shellability.   A pure shellable complex $\Delta$ is said to be \emph{extendably shellable} if any shelling of a pure subcomplex of $\Delta$ can be extended to a shelling of $\Delta$.  Here a subcomplex of $\Delta$ is a simplicial complex on the same vertex set as $\Delta$, whose set of facets is a subset of the facets of $\Delta$.  

Although shellable complexes arise naturally in many contexts, it seems that extendably shellable complexes are harder to come by. Results of Danaraj and Klee \cite{DanKle} imply that any $2$-dimensional triangulated sphere (which is necessarily polytopal) is extendably shellable, and Kleinschmidt \cite{Kle} has shown that any $d$-dimensional sphere with $d+3$ vertices is extendably shellable. Bj\"orner and Eriksson \cite{BjoEri} proved that independence complexes of rank 3 matroids are extendably shellable. On the other hand, Ziegler \cite{Zie} has shown that there exist simplicial 4-polytopes that are not extendably shellable. The inspiration for much of our work will be the following question posed by Simon \cite{Sim}.

\begin{conj}[Simon's conjecture] \cite{Sim}
The complex $\Delta_{n-1}^{(k)}$ is extendably shellable. 
\end{conj}

The $k=2$ case of Simon's conjecture follows from \cite{BjoEri} by considering the uniform matroid of rank $3$. More recently Bigdeli, Yazdan Pour, and Zaare-Nahandi \cite{BPZDec} established the $k \geq n-3$ cases, with a simpler proof provided independently by the second author \cite{Doc} based on results of Culbertson, Guralnik, and Stiller \cite{CulGurSti}. In \cite{CDGS} these last collection of authors also showed that if $\Delta$ is a $d$-dimensional simplicial complex on at most  $d+3$ vertices, then in fact the notions of shellable and extendably shellable are equivalent. This implies the $k = n-3$ case of Simon's conjecture and also provides a generalization of Kleinschmidt's results.  This result is also best possible in the sense that there are 2-dimensional complexes on 6 vertices that are not extendably shellable (see for example work of Moriyama and Takeuchi \cite{MorTak}, and also Bj\"orner \cite{Bjo}).  

Inspired by these observations, in this paper we consider a related notion. Recall that we use $\Delta_{n-1}^{(d)}$ to denote the $d$-skeleton of the $(n-1)$-dimensional simplex.

\begin{defn}
A pure $d$-dimensional simplicial complex $\Delta$ on $n$ vertices is said to be \defi{shelling completable} if there exists a shelling $F_1, F_2, \dots, F_s$ of $\Delta$ that can be taken as the initial sequence of some shelling of $\Delta_{n-1}^{(d)}$.
\end{defn}

In particular, a shelling completable complex is shellable. Note also that if $\Delta$ is shelling completable than \emph{any} shelling of $\Delta$ can be completed to a shelling of $\Delta_{n-1}^{(d)}$.  Indeed, suppose $\Delta$ is shelling completable and $F_1, \dots, F_s$ is a shelling order of $\Delta$ that can be completed to a shelling of $\Delta_{n-1}^{(d)}$ via $F_1, \dots, F_s, F_{s+1}, \dots F_t$.  For any $j > s$ the condition that $F_j$ must satisfy to constitute a shelling move is independent of the shelling order on $\Delta$.

Next note that Simon's conjecture is equivalent to the statement that any pure shellable complex is shelling completable.  From this perspective it is of interest to find a large class of shellable complexes that are shelling completable.  Also note that if $\Delta$ is a $k$-decomposable complex that can be completed (by adding facets) to $\Delta_{n-1}^{(d)}$ while maintaining $k$-decomposability, then $\Delta$ is shelling completable.

Our first examples of shelling completable complexes come from the class of pure shifted complexes, which are known to be vertex decomposable. Recall that a simplicial complex is \defi{shifted} if there exists an ordering of its vertex set $V = \{1, 2, \dots, n\}$ such that for any face $\{v_1, v_2, \dots, v_k\}$ replacing any $v_i$ with a smaller vertex results in a $k$-set that is also a face.  Note that if $\Delta$ is a pure $d$-dimensional shifted complex according to some ordering on its vertex set then adding the reverse-lexicographically (revlex) smallest missing $(d+1)$-subset $F$ again results in a shifted complex (see Proposition \ref{prop:shift} below).  This in turn implies that pure shifted complexes are in fact shelling completable. Our first result says that in fact any pure vertex decomposable complex admits an ordering of its vertex set with this property. 

\newtheorem*{thm:vdextend}{Theorem \ref{thm:vdextend}}
\begin{thm:vdextend}
Suppose $\Delta$ is a pure $d$-dimensional vertex decomposable complex on ground set $V$. Then either $\Delta$ is full over $V$ or there exists a linear order on $V$ such that if $F$ is the revlex smallest $(d+1)$-subset of $V$ not contained in $\Delta$ then the simplicial complex generated by $\Delta \cup \{F\}$ is vertex decomposable.
\end{thm:vdextend}

As a corollary we obtain a large class of shelling completable complexes, providing a positive answer to a weakened version of Simon's conjecture.

\newtheorem*{cor:VDcomplete}{Corollary \ref{cor:VDcomplete}}
\begin{cor:VDcomplete}
Pure vertex decomposable complexes are shelling completable.
\end{cor:VDcomplete}

Corollary \ref{cor:VDcomplete} also provides a new tool in the search for a counterexample to Simon's conjecture, since in particular it implies that any `stuck' initial shelling of $\Delta^{(k)}_{n-1}$ must in fact fail to be vertex decomposable.

Recall that the class of vertex decomposable complexes include pure shifted complexes and (independence complexes of) matroids. Theorem \ref{thm:vdextend} implies that there exists an ordering of the ground set of these complexes with the property that adding the revlex smallest missing $k$-subset results in a vertex decomposable complex. In the context of shifted complexes we have seen that the natural ordering of the ground set satisfies this property. For the case of matroids we prove that such \emph{revlex decomposing orders} (see Definition \ref{defn:decomporder}) are easy to come by.

\newtheorem*{prop:matroids}{Proposition \ref{prop:matroids}}
\begin{prop:matroids}
Let $\M$ be a rank $d$ matroid on ground set $V$. Then any ordering $v_1, v_2, \dots, v_n$ of $V$ with the property that $\{v_1, v_2, \dots,v_d\} \in \M$ is a revlex decomposing order.
\end{prop:matroids}

In particular for a rank $d$ matroid $\M$ it is `easy' to find a $d$-subset $F$ of the ground set with the property that $\M \cup \{F\}$, while no longer a matroid, is still vertex decomposable.  

Simon's conjecture and shelling completions are also related to certain notions of \emph{chordality} for simplicial complexes. In \cite{BigFar} Bigdeli and Faridi study a notion of \emph{ridge chordal} simplicial complexes (based on a similar notion for clutters introduced by Bigdeli, Yazdan Pour, and Zaare-Nahandi in \cite{BPZSta}). They conjecture that if $\Delta$ is a simplicial complex whose clique complex has shellable Alexander dual, then $\Delta$ is ridge-chordal.  They show that this conjecture implies Simon's conjecture, and also show that it holds true if one replaces the `shellable' assumption with `vertex decomposable'.   Recently Benedetti and Bolognini \cite{BenBol} have provided a counterexample to the full Bigdeli-Faridi conjecture.  We discuss these connections in Section \ref{sec:ridgechordal}.

In the last part of the paper we consider shelling completable complexes with few vertices (relative to dimension).  We exploit a connection between chordal graphs and certain shellable complexes to establish the following.

\newtheorem*{thm:shellableVD}{Theorem \ref{thm:shellableVD}}
\begin{thm:shellableVD}
Suppose $\Delta$ is a pure shellable $d$-dimensional simplicial complex on $d+3$ vertices.  Then $\Delta$ is vertex decomposable (and hence shelling completable).
\end{thm:shellableVD}

This theorem, along with results from \cite{CDGS}, imply that for $d$-dimensional complexes on at most $d+3$ vertices the notions of vertex decomposable, shellable, shelling completable, and extendably shellable are all equivalent.

The rest of the paper is organized as follows. In Section \ref{sec:VD} we recall some necessary definitions, prove Theorem \ref{thm:vdextend}, and discuss some corollaries and connections to chordal complexes. In Section \ref{sec:matroids} we discuss the notion of revlex decomposing orders in the context of matroids and prove Theorem \ref{prop:matroids}. In Section \ref{sec:few} we consider $d$-dimensional complexes on at most $d+3$ vertices and prove Proposition \ref{thm:shellableVD}. We end in Section \ref{sec:further} with some discussion and open questions.

\section{Vertex decomposable complexes and shelling completions}\label{sec:VD}

In this section we provide some necessary background regarding simplicial complexes and related notions, and also prove Theorem \ref{thm:vdextend}.  

A \defi{simplicial complex} $\Delta$ on a finite ground set $V$ is a collection of subsets of $V$ that is closed under taking subsets, so that if $\sigma \in \Delta$ and $\tau \subset \sigma$ then $\tau \in \Delta$.  The elements of $\Delta$ are called \defi{faces}. Note that we do not require $\{v\} \in \Delta$ for all $v\in V$. The elements $v \in V$ such that $\{v\} \in \Delta$ will be called the \defi{vertices} of $\Delta$, whereas elements $w \in V$ that are not vertices will be called \defi{loops}. In particular the vertex set of $\Delta$ can be a proper subset of its ground set. As in \cite{Jon} we adopt the convention that the \defi{void complex} $\emptyset$ is a simplicial complex, distinct from the \defi{empty complex} $\{\emptyset\}$.

 A \defi{facet} of $\Delta$ is an element that is maximal under inclusion. The \defi{dimension} of $\Delta$ is the largest cardinality (minus 1) of any facet. A simplicial complex $\Delta$ is \defi{pure} if all facets have the same cardinality. Given a ground set $V$ and a collection ${\mathcal S} = \{S_1, S_2, \dots, S_k\}$ of subsets $S_i \subset V$, we will use $\langle S_1, S_2, \dots, S_k\rangle$ to denote the simplicial complex \defi{generated by} ${\mathcal S}$, by which we mean the smallest simplicial complex containing the collection ${\mathcal S}$.  Note that if a simplicial complex $\Delta$ has facets $F_1, \dots, F_k$ then $\Delta = \langle F_1, \dots, F_k \rangle$. With these notions we can state the definition of a shellable complex.


\begin{defn}\label{defn:shellable}
 A pure $d$-dimensional simplicial complex $\Delta$ is  \defi{shellable} if there exists an ordering of its facets $F_1, F_2, \dots, F_s$ such that for all $k = 2,3, \dots, s$ the simplicial complex
\[\left ( \bigcup_{i=1}^{k-1} \langle F_i \rangle \right ) \cap \langle F_ k \rangle\]
\noindent
is pure of dimension $d-1$.  By convention the void complex $\emptyset$ and the empty complex $\{\emptyset\}$ are both shellable.  
\end{defn}

Note that a shellable complex is connected as long as $d \geq 1$, and a pure $1$-dimensional simplicial complex (a graph) is shellable if and only if it is connected. We next recall the notion of link, star, and deletion of a face in a simplicial complex.

\begin{defn}
Suppose $\Delta$ is a simplicial complex on ground set $V$ and let $F \in \Delta$ be a nonempty face.  The  \defi{link},  \defi{star} and the \defi{deletion} of $F$ are defined as 
$$\lkc(F) := \{G \in \Delta : G \cap F = \emptyset, G \cup F \in \Delta\},$$
$$\stc(F) := \{G \in \Delta : F \subset G\},$$
$$\delc(F) := \{G \in \Delta: F \nsubseteq G\}.$$
The ground set of $\stc(F)$ is $V$, whereas the ground set of $\lkc(F)$ is $V \backslash F$.  If $|F| > 1$ then $\delc(F)$ has ground set $V$, and if $F = \{v\}$ then the ground set is $V \backslash \{v\}$.   
\end{defn}

We note that shellability is preserved by taking links, a fact that will be useful later.

\begin{lem}[\cite{ZieLec}, Lemma 8.7]
\label{lem:linkshell}
If $\Delta$ is a shellable simplicial complex and $F \in \Delta$ is any face, then the link $\lkc(F)$ is shellable.
\end{lem}

We next define the class of vertex decomposable simplicial complexes recursively as follows.

\begin{defn}
\label{def:vd}
A simplicial complex $\Delta$ is \defi{vertex decomposable} if $\Delta$ is a simplex (including $\emptyset$ and $\{\emptyset\}$), or $\Delta$ contains a vertex $v$ such that
\begin{enumerate}
    \item both $\delc(v)$ and $\lkc(v)$ are vertex decomposable, and
    \item any facet of $\delc(v)$ is a facet of $\Delta$.
\end{enumerate}
A vertex $v$ that satisfies the second condition is called a \defi{shedding vertex} of $\Delta$. We will call a vertex $v$ that satisfies both conditions a \defi{decomposing vertex}.
\end{defn}


 Vertex decomposable complexes were introduced in the pure setting by Provan and Billera \cite{ProBil} and extended to non-pure complexes by Bj\"orner and Wachs \cite{BjoWac}.  It is known that any pure vertex decomposable complex is shellable, a fact implied by the following result of Wachs \cite{Wac}.

\begin{lem}[\cite{Wac}, Lemma 6]
Suppose $\Delta$ is a simplicial complex with shedding vertex $v$.  If both $\delc(v)$ and $\lkc(v)$ are shellable then $\Delta$ is shellable.
\end{lem}

In what follows we will restrict ourselves to pure simplicial complexes.  Given two facets $F$ and $G$ of a pure $d$-dimensional simplicial complex $\Delta$ we say that $F$ and $G$ are \defi{adjacent} if they differ by one vertex, i.e. $|F \cap G| = d$.

\begin{lem}
\label{lem:shedd}
Suppose $\Delta$ is a pure simplicial complex and let $v \in V$ be a vertex.  Then $v$ is a shedding vertex if and only if any facet of $\stc(v)$ is adjacent to some facet of $\delc(v)$.
\end{lem}
\begin{proof}
For one direction suppose $v$ is a shedding vertex of $\Delta$ and let $F$ be a facet of $\stc(v)$. We have that $F \backslash \{v\}$ is a face of $\delc(v)$. Let $F'$ be a facet of $\delc(v)$ containing $F \backslash \{v\}$. Since $v$ is a shedding vertex we have that $F'$ is a facet of $\Delta$, and since $\Delta$ is pure we have that $F'$ and $F$ have the same cardinality. We conclude that $F'$ is a facet of $\delc(v)$ adjacent to $F$.  

For the other direction suppose $v$ is not a shedding vertex, so that some facet $F$ of $\delc(v)$ is not a facet of $\Delta$.  Let $F^\prime$ be a facet of $\Delta$ that contains $F$, so that $v \in F^\prime$ and hence $F^\prime$ is a facet of $\stc(v)$.  But $F^\prime$ cannot be adjacent to any facet in $\delc(v)$ since $F$ was a facet of $\delc(v)$.
\end{proof}

We next turn to the question of shelling completions for vertex decomposable complexes.  For this we will need the following concepts. Suppose $\Delta$ is a $d$-dimensional simplicial complex on ground set $V$.  We say that $\Delta$ is \defi{full (over $V$)} if it is the $d$-skeleton of the simplex over the vertex set $V$, i.e. it consists of all $(d+1)$-subsets of $V$.  Note that a $d$-simplex is full if and only if $|V| = d+1$.

We will also need the notion of reverse lexicographic (revlex) order on $k$-subsets of an ordered ground set. For this recall that if $V = \{1,2, \dots, n\}$ is a linearly ordered set, then $\{v_1 < v_2 < \cdots < v_k\}$ is \defi{revlex smaller} than $\{w_1 < w_2 < \cdots < w_k\}$ if for the largest $j$ with $v_j \neq w_j$ we have $v_j < w_j$.  Note that if one adds a new element $n+1$ to the set $V$ then any $k$-subset that contains $n+1$ will be revlex larger than any $k$-subset that does not contain $n+1$. This implies the following observation.

\begin{lem}\label{lem:revlex}
Suppose $V$ is a finite linearly ordered set with largest element $v$, and suppose ${\mathcal F}$ is a collection of $d$-subsets of $V$. Let ${\mathcal G} \subset {\mathcal F}$ denote the collection of those $d$-subsets $F \in {\mathcal F}$ such that $v \notin F$.  Then assuming ${\mathcal G} \neq \emptyset$, the revlex smallest element of ${\mathcal G}$ is also the revlex smallest element of ${\mathcal F}$. 
\end{lem}

As mentioned in Section \ref{sec:intro}, we can use revlex orders to build new shifted complexes from existing ones. More precisely we have the following.

\begin{prop}\label{prop:shift}
Suppose $\Delta$ is a pure shifted $d$-dimensional simplicial complex with respect to some linear order on its ground set $V$, and assume that $\Delta$ is not full on $V$. Let $F$ be the revlex smallest $(d+1)$-subset of $V$ satisfying $F \notin \Delta$. Then the complex $\Delta \cup \langle F \rangle$ is again a pure shifted simplicial complex.
\end{prop}

\begin{proof}
One can see that replacing any element $x \in F$ with some $y \in V$ satisfying $y < x$ results in a $(d+1)$-subset $F^\prime = (F \backslash \{x\}) \cup \{y\}$ that is revlex smaller than $F$.  Hence $F^\prime$ is a facet of $\Delta \cup \langle F \rangle$, implying that $\Delta \cup \langle F \rangle$ satisfies the shifted condition for all sets of size $(d+1)$.  But since $\Delta \cup \langle F \rangle$ is pure this implies the condition for all subsets, so that $\Delta \cup \langle F \rangle$ is shifted.
\end{proof}

Our main result generalizes this observation for the class of vertex decomposable complexes.


\begin{thm}
\label{thm:vdextend}
Suppose $\Delta$ is a pure $d$-dimensional vertex decomposable simplicial complex on ground set $V$. Then either $\Delta$ is full over $V$ or there exists a linear order on $V$ such that if $F$ is the revlex smallest $(d+1)$-subset of $V$ not contained in $\Delta$ then the simplicial complex $\Delta \cup \langle F \rangle$ is vertex decomposable.
\end{thm}

\begin{proof}
We use induction on $d$, then on $n = |V|$. Note that for $d=-1$ the statement is trivial. If $d=0$ then $\Delta$ is either full (it consists of all $\{v\}$ for $v\in V$) or adding any missing element in $V$ results in a vertex decomposable complex. For any $d \geq 1$, the case $n=d+1$ is true since $\Delta$ is a simplex on $V$ and hence is full.   

We now assume that $d \geq 1$ and $n \geq d+2$. We claim that we can also assume that $\Delta$ is not full over its vertex set $W$.  To see this suppose that $\Delta$ is full over $W$ and pick any ordering on the ground set  $V$ so that all elements in $W$ are smaller than all elements in $V \backslash W$.   Let $F^\prime$ be the revlex smallest facet of $\Delta$ according to this ordering, and let $w$ be the largest element of $F^\prime$. Now let $v$ be the smallest element in $V \backslash W$ and set $F = (F^\prime \backslash \{w\}) \cup \{v\}$. Then the simplicial complex $\Delta^\prime = \Delta \cup \langle F \rangle$ is vertex decomposable since $\delcp(v) = \Delta$, the complex $\lkcp(v)$ is a simplex, and $v$ is a shedding vertex since $F$ is adjacent to $F^\prime$.  We conclude that $v$ is a decomposing vertex.



Now suppose that $\Delta$ is a vertex decomposable complex that is not full on its vertex set. At this point we can continue the induction by picking a vertex $v$ which is a loop (if one exists) or a decomposing vertex (which always exists sicne $\Delta$ is vertex decomposable.  The reason for this flexibility will become explained later.

{\bf Case 1.} Suppose we choose $v$ to be a loop. In this case we have by induction on $n$ that the complex $\Delta$ on ground set $V \backslash \{v\}$ admits the desired ordering. To extend the ordering to all of $V$ we simply define $v$ to be the largest element. This provides the desired ordering of $V$ by Lemma \ref{lem:revlex}.

{\bf Case 2.} Next suppose we choose $v$ be a decomposing vertex of $\Delta$, so that both $\delc(v)$ and $\lkc(v)$ are vertex decomposable. Consider the partition of facets of $\Delta$ into $\stc(v) \dot\cup \delc(v)$.  We now have two subcases. 

{\bf Case 2a.} First assume $\delc(v)$ is not full on its vertex set $V \backslash \{v\}$. By induction on $n$, we have that $\delc(v)$ admits an ordering of its vertex set $V \backslash \{v\}$ such that the revlex smallest missing $(d+1)$-subset $F \subset V \backslash \{v\}$ can be added to obtain another vertex decomposable complex. Let $\Delta^\prime = \Delta \cup \langle F \rangle$. Note that $\delcp(v) = \langle \delc(v) \cup \{F\} \rangle$ and $\lkcp(v) = \lkc(v)$, both of which are vertex decomposable by assumption.  We have not added any facets to $\stc(v)$, so $v$ is a shedding vertex of $\Delta^\prime$ by Lemma~\ref{lem:shedd}. Hence $v$ is a decomposing vertex for $\Delta^\prime$.  Use the ordering on $V \backslash \{v\}$ provided by the induction hypothesis and extend to all of $V$ by declaring $v$ to be the largest element. This again provides the desired ordering of $V$ by Lemma \ref{lem:revlex}.

{\bf Case 2b.} We next suppose that $\delc(v)$ is full on vertex set $V \backslash \{v\}$. Note that we can assume that $\lkc(v)$ is not full on $V \backslash \{v\}$ since otherwise this would imply that $\Delta$ is full. In this case, by induction on $d$ we have that there exists an ordering on the ground set $V \backslash \{v\}$ of $\lkc(v)$ with the desired properties.  In particular the revlex smallest missing $d$-set $G \subset V \backslash \{v\}$ has the property that the simplicial complex $\langle \lkc(v) \cup \{G\} \rangle$ is vertex decomposable.  Now let $\Delta^\prime$ denote the simplicial complex generated by $\Delta \cup \{G \cup \{v\}\}$. Note that $\delcp(v) = \delc(v)$ since $\delc(v)$ is full.  Also $\lkcp(v) = \langle \lkc(v) \cup (G \backslash \{v\}) \rangle$, which we have assumed is vertex decomposable.

Finally we claim that $v$ is a shedding vertex. To see this note that there has to be some facet in $\delc(v)$ that is adjacent to $G \cup \{v\}$ since $\delc(v)$ is full. The claim then follows from Lemma~\ref{lem:shedd}.  Once again we use the ordering on $V \backslash \{v\}$ obtained by induction and extend to all of $V$ by declaring $v$ to be the largest element.  One can see that $G \cup \{v\}$ is the revlex smallest missing $(d+1)$-subset of $\Delta$ because we are assuming that $\delc(v)$ is full. The result follows.
\end{proof}

For an illustration of the various steps in the above proof, we refer to Example~\ref{ex:decompose}. To establish our desired corollary we will next need the following observation, for which we thank Michelle Wachs for a simplified proof.

\begin{lem}
\label{lem:bruno}
Suppose $\Delta$ is a shellable $d$-dimensional complex on ground set $V$ and suppose $F$ is a  $(d+1)$-subset of $V$ with the property that the complex $\Delta^\prime = \langle \Delta \cup \{F\} \rangle$ is again shellable.  Then any shelling of $\Delta$ can be extended to a shelling of $\Delta^\prime$ by adding $F$ as the last facet.
\end{lem}

\begin{proof}
Suppose $\Delta$, $F$, and $\Delta^\prime$ are as above and let $\Gamma = \langle F \rangle \cap \Delta$. First note that it is enough to prove that $\Gamma$ is pure of dimension $d-1$, since then adding $F$ to any shelling of $\Delta$ satisfies the conditions of Definition \ref{defn:shellable}.
Suppose $F_1, F_2, \dots, F_k$ is a shelling of $\Delta^\prime$ where $F = F_t$.  For a contradiction suppose that $\Gamma$ has an inclusion-wise maximal face $G$ of dimension less than $d-1$.  Then $G \subset F_t \cap F_i$ for some $i \neq t$. Let $j$ be the smallest such $i$.

{\bf Case 1.} First suppose $j <t$.  Then we have $G \in \langle F_t \rangle \cap \left( \cup_{i=1}^{t-1} \langle F_i \rangle \right)$.  Since $F_1, \dots, F_k$ is a shelling we have a $(d-1)$-dimensional $H$ such that $G \subsetneq H$ and 
\[H \in \langle F_t \rangle \cap \left( \bigcup_{i=1}^{t-1} \langle F_i \rangle \right) \subset \langle F_t \rangle \cap \left( \bigcup_{i=1}^{k} \langle F_i \rangle \right)= \Gamma.\]
\noindent
This is a contradiction to the assumption that $G$ is a facet of $\Gamma$.

{\bf Case 2.} Next suppose $j > t$. In this case $F_t$ and $F_j$ are the only facets among the $F_1, \dots, F_j$ that contain $G$. But since $\langle F_j \rangle \cap \left( \cup_{i=1}^{j-1} \langle F_i \rangle \right)$ is pure $(d-1)$-dimensional it follows that there exists a $(d-1)$-dimensional face $H$ such that 
\[G \subsetneq H = F_j \cap F_t \in \langle F_t \rangle \cap \left( \bigcup_{i=1}^{k} \langle F_i \rangle \right) = \Gamma,\]
again a contradiction to the fact that $G$ is a facet.
\end{proof}

As a corollary we get a large class of complexes that are shelling completable, and hence we obtain a weakened form of Simon's conjecture.
\begin{cor}\label{cor:VDcomplete}
Pure vertex decomposable complexes are shelling completable.
\end{cor}

\begin{proof}
Suppose $\Delta$ is pure $d$-dimensional vertex decomposable on ground set $V$, where $|V| = n$. Let $m$ be the number of $(d+1)$-subsets of $V$ that are missing as facets in $\Delta$. If $m=0$ then $\Delta = \Delta^{(d)}_{n-1}$ is full and we are done.  Otherwise by Theorem \ref{thm:vdextend} we have some $(d+1)$-subset $F \subset V$ such that $F \notin \Delta$ with $\Delta \cup \{F\}$ vertex decomposable, and hence shellable. From Lemma \ref{lem:bruno} we know that any shelling order of $\Delta$ can be extended to a shelling of $\Delta \cup \{F\}$. The result follows by induction on $m$.
\end{proof}


\subsection{Connection to ridge chordal complexes} \label{sec:ridgechordal}

As mentioned in Section \ref{sec:intro}, our results also connect to notions of higher dimensional \emph{chordality} for simplicial complexes. In this section we explain this relation, after reviewing the relevant definitions. 

For a pure $d$-dimensional simplicial complex $\Delta$ on vertex set $V$, a \emph{clique} is a subset $K \subset V$ such that either $|K| < d+1$ or else all $(d+1)$-subsets of $K$ appear among the facets of $\Delta$.  A \emph{ridge} $R \subset V$ of $\Delta$ is a face of $\Delta$ such that $|R| = d$.  The \emph{closed neighborhood} of a ridge $R$ is the set
\[N_\Delta[R] = R \cup \{v \in V: \text{$R \cup \{v\}$ is a facet of $\Delta$}\}.\] 
We say that a ridge $R$ is \emph{simplicial} if $N_\Delta[R]$ is a clique.

In \cite{BPZSta} Bigdeli, Yazdan Pour, and Zaare-Nahandi use these concepts to define a notion of chordal compelxes (which we will refer to as \emph{ridge chordal} to distinguish it from other notions). In what follows we use $\Delta \backslash R$ to denote the simplicial complex consisting of all faces $\sigma \in \Delta$ that do not contain $R$.

\begin{defn}
A pure simplicial complex $\Delta$ is \emph{ridge chordal} if $\Delta = \emptyset$ or $\Delta$ admits a simplicial ridge $R$ such that $\Delta \backslash R$ is ridge chordal.
\end{defn}

In \cite{BPZSta} the authors use the language of \emph{clutters} but our definition is equivalent for the case of pure simplicial complexes.  Also note that a pure $1$-dimensional simplicial complex is ridge chordal if and only if it is a chordal graph. In \cite{BigFar}, Bigdeli and Faridi extend this notion to the setting of arbitrary simplicial complexes.

For a simplicial complex $\Delta$ we let $\Cl(\Delta)$ denote the \emph{clique complex} of $\Delta$, by definition the simplicial complex whose faces are the cliques of $\Delta$. Note that $\Cl(\Delta)$ is a simplicial complex of dimension at least $d$, with the same $d$-faces as $\Delta$ and a complete $k$-skeleton for all $k < d$. The following conjecture has appeared in \cite{Doc} and \cite{Nik} (and in a stronger form in \cite{BigFar}).  In what follows we use $\Gamma^*$ to denote the \emph{Alexander dual} of a simplicial complex $\Gamma$.

\begin{conj} \label{conj:chordal}
If $\Delta$ is a pure simplicial complex such $(\Cl(\Delta))^*$ is shellable, then $\Delta$ is ridge-chordal.
\end{conj}

In \cite{Nik} Nikseresht shows that if $(\Cl(\Delta))^*$ is vertex-decomposable, then $\Delta$ is ridge-chordal. From results of \cite{BPZSta} this implies that if $\Delta$ is vertex-decomposable and not full, then there exists a $(d+1)$-subset $F$ such that $F \notin \Delta$ and such that $\Delta \cup \{F\}$ is shellable.  In other words the shelling can be continued for one more step. Our Corollary \ref{cor:VDcomplete} says that in fact the shelling can be completed to the full $d$-skeleton.

Recently in \cite{BenBol} Benedetti and Bolognini provided a counterexample to Conjecture \ref{conj:chordal}, in the form of a complex $\Delta$ is that is not ridge-chordal but such that $(\Cl(\Delta))^*$ is shellable (in fact they provide an infinite family of such complexes). We suspect that the complexes in this family are in fact shelling completable although it does not seem easy to check.  It is also an open question whether one can construct a non ridge-chordal complex $\Delta$ such that $(\Cl(\Delta))^*$ is $1$-decomposable (see Section \ref{sec:further}).

In the context of ridge-chordality, we see that the condition of being vertex decomposable is strong enough to imply results that are not true under the weaker shellability assumption.  It is still an open question whether this dichotomy exists in the context of shelling completability.  See Section \ref{sec:further} for further discussion.

\section{Revlex decomposing orders and matroid complexes}\label{sec:matroids}

Recall that a simplicial complex $\M$ is a \defi{matroid} if it is pure and its set of facets satisfy the following exchange property: If $F$ and $G$ are facets of $\M$ then for any $x \in F \backslash G$ there exists  some $y \in G \backslash F$ such that $(F \backslash \{x\}) \cup \{y\}$ is a facet of $\M$. The facets of $\M$ are usually called \defi{bases} in this theory. Also note that in some contexts this simplicial complex is called the \emph{independence complex} of $\M$ but we will sometimes simply refer to it as the matroid itself. It is well known that matroids are vertex decomposable \cite{ProBil} and hence Corollary \ref{cor:VDcomplete} implies the following.

\begin{cor}
Independence complexes of matroids are shelling completable.
\end{cor}

Given that any shelling of a rank $d$ matroid can be completed to a shelling of the full skeleton $\Delta_n^{(d-1)}$, a natural question to ask is whether one can control which facet can be added in the next step of the completion. In the context of matroids, one expects some flexibility since matroids themselves admit many shelling orders.   In particular recall that if $V = \{v_1, v_2, \dots, v_n\}$ is any ordering of the ground set of a rank $d$ matroid $\M$, then both lexicographic (\cite{Bjo} Theorem 7.3.4) and reverse lexicographic (\cite{GuoSheWu} Proposition 6.3) orderings of the facets (bases) of $\M$ give rise to a shelling of the complex $\M$.  For our purposes we will need the following notion.


\begin{defn}\label{defn:decomporder}
Suppose $\Delta$ is a a pure $d$-dimensional vertex decomposable simplicial complex.  An ordering $v_1, v_2, \dots, v_n$ of its ground set is a \defi{revlex decomposing order} for $\Delta$ if the complex generated by $\Delta \cup \{F\}$ is again vertex decomposable, where $F$ is the revlex smallest $(d+1)$-subset of $V$ that is missing from $\Delta$.
\end{defn}

Note that Theorem \ref{thm:vdextend} says that any vertex decomposable complex admits a revlex decomposing order. From the inductive proof of that theorem we get the following recursive method for detecting whether a given ordering of the ground set is a revlex decomposing order.

\begin{cor}
\label{cor:DO}
Suppose $\Delta$ is a pure $d$-dimensional simplicial complex. A linear ordering $v_1,\ldots,v_n$ of its ground set is a revlex decomposing order for $\Delta$ if one of the following is true:
\begin{itemize}
    \item $v_n$ is a loop and $v_1,\ldots,v_{n-1}$ is a revlex decomposing order for $\Delta$ (on ground set $v_1, \ldots, v_{n-1}$),
    \item $v_n$ is a decomposing vertex for $\Delta$, $\del_{\Delta}(v_n)$ is not full, and $v_1,\ldots,v_{n-1}$ is a revlex decomposing order for $\del_{\Delta}(v_n)$,
    \item $v_n$ is a decomposing vertex for $\Delta$, $\del_{\Delta}(v_n)$ is full, and $v_1,\ldots,v_{n-1}$ is a revlex decomposing order for $\lk_{\Delta}(v_n)$,
    \item $\Delta$ is full over its vertex set and any loop is bigger than any non-loop in this ordering.
    \end{itemize}
\end{cor}

Also note that in the proof of Theorem \ref{thm:vdextend} at each step in the induction we must choose a vertex $v$ where we employ the inductive hypothesis on either the deletion $\delc(v)$ (in the first case) or the link $\lkc(v)$ (in the latter). In this way we can obtain a sequence of subcomplexes
\[\Delta = \Delta_n, \Delta_{n-1}, \dots, \Delta_s,\]
\noindent
where $n$ is the size of the ground set of $\Delta$, and $\Delta_s$ is a simplex over some (possibly smaller) ground set. We illustrate this process below with a worked example.  


\begin{example}
\label{ex:decompose}
Suppose $\Delta$ is the $3$-dimensional simplicial complex on ground set $\{1,2,\dots, 7\}$ with facets
\[\{1234, 1235,1245,1345,2345,1236,1246,1256,2356,1237, 2347\}.\]
Here we abuse notation and for example let $1245$ denote the $4$-subset $\{1,2,4,5\}$. We use the natural ordering on the ground set and verify that it is a revlex decomposing order.  In what follows, we will describe the various simplicial complexes in terms of their generating sets of facets.

First we define $\Delta_7 = \Delta$ and note that 
\[\text{del}_{\Delta_7}(7) = \langle 1234, 1235,1245,1345,2345,1236,1246,1256,2356\rangle\]
\noindent
is not full (e.g. $1346$ is missing).  Hence at this step we consider the deletion of vertex $7$ and define $\Delta_6 = \del_{\Delta_7}(7)$.

Next we note that $\text{del}_{\Delta_6}(6)$ is full so we now consider the link of $6$ in $\Delta_6$ and define
\[\Delta_5 = \lk_{\Delta_6}(6) = \langle 123,124,125,235\rangle.\]

Continuing in this fashion we have that $\text{del}_{\Delta_5}(5)$ is not full so we define
\[\Delta_4 = \del_{\Delta_5}(5) = \langle 123,124\rangle.\]

Next we see that $\text{del}_{\Delta_4}(4) = \langle 123\rangle$ is full and so we consider the link of $4$ in $\Delta_4$ and have
\[\Delta_3 = \lk_{\Delta_4}(4) = \langle 12 \rangle.\]

At this point we see that $\Delta_3$ is a simplex (on ground set $\{1,2,3\}$) and hence we have reached a base case.

Reversing this process, we see that at each step the addition of a new facet $F$ leads to a vertex decomposable complex. We begin with $\Delta_3$ since it is full over its vertex set, and hence a base case of Theorem \ref{thm:vdextend}. We extend $\Delta_3$ to $\Delta_3'$ by noting that $3$ is the smallest loop, and $2$ is the largest vertex.  Hence we add the facet $(12 \backslash \{2 \}) \cup \{3\} = 13$. Now in $\Delta_4$ we replace $\lk_{\Delta_4}(4) = \Delta_3$ with $\Delta_3'$, which results in adding the facet $134$ to obtain $\Delta_4'$. In $\Delta_5$ we replace $\del_{\Delta_5}(5) = \Delta_4$ with $\Delta_4'$ which results in adding the facet $134$ to obtain $\Delta_5'$. Next in $\Delta_6$ we replace $\lk_{\Delta_6}(6) = \Delta_5$ with $\Delta_5'$, adding facet $1346$ to obtain $\Delta_6'$.  Finally in $\Delta_7$ we replace $\del_{\Delta_7}(7) = \Delta_6$ with $\Delta_6'$, adding the facet $1346$. We note that the simplicial complex $\langle \Delta \cup \{1346\} \rangle$ is vertex decomposable, and $1346$ is indeed the smallest element missing from $\Delta = \Delta_7$ among the revlex ordered $4$-subsets of $\{1,\ldots,7\}$.

\end{example}


We will use the above observations to show that many orderings of the ground set of a matroid give rise to revlex decomposing orders.  For this we will need the following result.




\begin{lem}\label{lem:matroid}
Let $\M$ be a rank $d$ matroid on ground set $V$ and suppose that $\M$ is full on $V \backslash \{v\}$ for some nonloop $v \in V$. Suppose $F$ is a $d$-subset of $V$ with $v \in F$ such that $F$ is not a facet (basis) of $\M$. Then the complex generated by $\M \cup \{F\}$ is again a matroid.
\end{lem}

\begin{proof}
We first observe that for any $d$-subset $A$ of $V \backslash \{v\}$ there exists a $(d-1)$-subset $B \subset A$ with the property that $B \cup \{v\}$ is a facet of $\M$.  This follows from applying the exchange property with $A$ and any facet $F$ that contains $v$ (which must exist since $v$ is not a loop).

Now let $F = \{x_1, x_2, \dots, x_{d-1},v\}$ be a $d$-subset that is missing from $\M$. Let $G$ be any facet of $\M$.  We will verify the exchange properties between $F$ and $G$. We first claim that for any $f \in F$ there exists some $g_i$ such that $(G \backslash \{g_i\}) \cup \{f\}$ is a facet of $\M$. If $f \neq v$ then this is clear since $\M$ is full on $V \backslash \{v\}$. If $f = v$ we use the above observation with $A = G$ to obtain $B = G \backslash \{g_i\}$.

Next we claim that for any $g \in G$ with $g \neq v$ there exists some $x_i \in F$ such that $(F \backslash \{x_i\}) \cup \{g\}$ is a facet of $\M$. For this we again use the above observation with $A = \{g, x_1,...,x_{d-1}\}$. Note that the resulting $B$ must contain $g$ since otherwise $F$ would be a facet of $\M$. We take $x_i = A \backslash B$ and the result follows.
\end{proof}

\begin{prop}
\label{prop:matroids}
Let $\M$ be a rank $d$ matroid and suppose $v_1, v_2, \dots, v_n$ is any linear ordering of its ground set $V$ with the property that $\{v_1, v_2, \dots,v_d\} \in \M$. Then $v_1, v_2, \dots, v_n$ is a revlex decomposing order for $\M$.
\end{prop}

\begin{proof}
Suppose that $v_1, v_2, \dots, v_n$ is such an ordering of $V$ and let $F$ be the revlex smallest $d$-subset of $V$ that is missing from $\M$. Let $q \geq 2$ be the smallest index such that $v_q \in F$ and $v_{q-1} \not \in F$. 

Since $\M$ is a matroid we have that any element of $V$ is a loop or a decomposing vertex.   Repeatedly using Corollary~\ref{cor:DO}, we get that $v_1, v_2, \dots,v_n$ is a revlex decomposing order for $\M$ if $v_1,v_2, \dots,v_q$ is a revlex decomposing order for
\[\Gamma := \lk_{\Delta}(F \cap \{v_{q+1},\ldots,v_n\})|_{\{v_1,\ldots,v_q\}}.\]

We see that $\Gamma$ is a matroid (since it is obtained via a sequence of deletions and links of a matroid), with the property that $\del_{\Gamma}(v_q)$ is full.  This follows from the fact that $F$ was the revlex smallest $d$-subset missing from $\M$.  From Lemma \ref{lem:matroid} we see that any ordering of the ground set of $\Gamma$ is a revlex decomposing order. The result follows.
\end{proof}

As a consequence of Proposition \ref{prop:matroids} we get the following.

\begin{cor}\label{cor:matroidorder}
Let $\M$ be a rank $d$ matroid and suppose $v_1, v_2, \dots, v_n$ is any linear ordering of its ground set $V$ with the property that $\{v_1, v_2, \dots,v_d\} \in \M$, and let $F$ be the revlex smallest $d$-subset missing from $\M$.  Then the complex generated by $\M \cup \{F\}$ is vertex decomposable.
\end{cor}

\begin{rem}
Given a rank $d$ matroid $\M$ on ground set $V$, a related question to ask is whether there exists a $d$-subset $F \subset V$ such that $\M \cup \{F\}$ is again a \emph{matroid}. Let ${\mathcal B}(M)$ denote the collection of facets (bases) of a matroid. If $F$ happens to be a \emph{circuit-hyperplane} (that is $F$ is a circuit of $\M$ and $V \backslash F$ is a circuit of the dual matroid $\M^*$), then Kahn \cite{Kah} has shown that ${\mathcal B}(\M) \cup \{F\}$ forms the set of bases of a new matroid $\M^\prime$. This process is known as a \emph{relaxation} of $\M$. Furthermore, Truemper \cite{Tru} has shown that if $M_1$ and $M_2$ are connected matroids on the same ground set and the symmetric difference ${\mathcal B}_1 \triangle {\mathcal B}_2$ has cardinality one, then in fact one of ${\mathcal B}_1$ and ${\mathcal M}_2$ is obtained from the other by relaxing a circuit-hyperplane.

As for a nonexample we thank Jay Schweig for the following explicit construction: Consider the rank $3$ matroid $\M$ with ground set $\{x,a,b,c,d\}$ generated by the facets (bases) $\{xab, xac, xad, xbc, xbd, xcd\}$ (so that $M$ is the cone over the rank 2 uniform matroid on 4 elements). Note that if $F$ is any missing $3$-subset of $M$ then $x \notin F$. Now if $F$ had the exchange property with $xab$ it could not with $xcd$, since we cannot replace $x$ with anything to obtain a basis of $M \cup \{F\}$.
\end{rem}

\begin{rem}
A \defi{shedding order} for a vertex decomposable complex $\Sigma$ is an ordering $v_1, v_2, \dots, v_n$ of its vertices with the property that $v_{i}$ is a decomposing vertex of $\Sigma_i$, where we define $\Sigma_n = \Sigma$ and $\Sigma_{i-1} = \del_{\Sigma_{i}}(v_i)$ (so that $\Sigma_{i-1}$ is necessarily vertex decomposable). In \cite[Proposition 3.2.3]{ProBil}  it is shown that any ordering of the vertices of a matroid $\M$ is a shedding order for $\M$. 


We note that there exist shedding orders that are not revlex decomposing orders. For instance consider the simplicial complex $\Gamma$ with facets
\[\{1234,1235,1245,1345,2345,1346,1456,2456,3456\},\]
which is the complex obtained from $\Delta_6$ of Example~\ref{ex:decompose} by reversing $\{1,2,3,4,5\}$.  Since $\del_\Sigma(6)$ is full, we have that $1,2,\dots,6$ is a shedding order. But in $\lk_\Sigma(6) = \{134,145,245,345\}$ we cannot use $5$ as a decomposing vertex, so that $1,2,\dots,6$ is not a revlex decomposing order. On the other hand $5,4,3,2,1,6$ gives a revlex decomposing order, as discussed above.
\end{rem}

In light of Corollary \ref{cor:matroidorder} a natural question to ask is whether there exists an ordering of the ground set such that \emph{all} missing facets can be added in reverse lexicographic order.  The next example shows that an arbitrary ordering will not work.

\begin{example}
Let $\M$ be the matroid on ground set $[6]$ generated by the facets $\{1234, 1345,2346, 3456\}$.  We note that adding the revlex smallest missing $4$-subset $1235$ results in a shelling move but the shelling fails when we continue to add the next revlex smallest subsets
\[1235, 1245, 1236, 1246, 1256.\]
To see this let $\Delta$ denote the complex obtained by adding these $4$-subsets and consider $F = 56$, a face of $\Delta$.  We note that $\lkc(F) = \{34, 12\}$, which is $1$-dimensional and disconnected and hence not shellable. By Lemma \ref{lem:linkshell} we conclude that $\Delta$ is not shellable.
 \end{example}


\section{Complexes with few vertices}\label{sec:few}

In \cite{CDGS} it is shown that a $d$-dimensional complex $\Delta$ on $d+3$ vertices is extendably shellable if and only if $\Delta$ is shellable. In this section we show that these conditions are also equivalent to $\Delta$ being vertex decomposable.  In what follows we will assume that our complexes have no loops, so that the vertex set of $\Delta$ coincides with its ground set.

For our result we will exploit a connection between shellable complexes and the notion of a chordal graph.  Recall that a simple graph $G$ is \defi{chordal} if it has no induced cycles of length 4 or more (so that all cycles of length 4 or more have a `chord').   It is well known that any chordal graph admits a \defi{simplicial vertex}, a vertex $v \in V(G)$ such that its neighborhood (the subgraph induced on the set of vertices adjacent to $v$) is a complete graph.  From \cite{Doc} we have the following result, adapted for our purposes.

\begin{lemma} \label{lem:exposed}
Let $K_n$ denote the complete graph on vertex set $[n] = \{1, 2, \dots, n\}$.  Suppose $\{e_1, e_2, \dots, e_k\} \subset E(K_n)$ is a collection of edges and for each $j = 1,2, \dots, k$ let $F_j = [n] \backslash e_j$ denote the complementary $(n-2)$-subset. Then $K_n \backslash \{e_1, e_2, \dots, e_k\}$ is a chordal graph if and only if the simplicial complex generated by $F_1, F_2, \dots, F_k$ is shellable.
\end{lemma}

\begin{lemma}\label{lem:graphVD}
Any shellable 1-dimensional simplicial complex (graph) $G$ is vertex decomposable. 
\end{lemma}

\begin{proof}
Recall that a graph $G$ is shellable if and only if it is finite and connected. Let $G$ be a connected graph on vertex $[n]$. If $n=1$ then $G$ is a simplex and hence vertex decomposable.  Suppose $n \geq 2$ and let $T$ be a spanning tree of $G$.  Choose $v \in G$ to be a leaf vertex of $T$.  Then $G \backslash \{v\}$ is connected and $\text{lk}_G(v)$ is nonempty. The result follows by induction on $n$.
\end{proof}

\begin{lemma}\label{lem:smallVD}
Suppose $\Delta$ is a pure $d$-dimensional simplicial complex on at most $d+2$ vertices. Then $\Delta$ is vertex decomposable.
\end{lemma}

\begin{proof}
We use induction on $d$.  We see that the statement is true for $d=1$ by Lemma \ref{lem:graphVD}, since a graph without isolated vertices on at most $3$ vertices must be connected (and hence shellable).  For $d \geq 2$, if $\Delta$ is a simplex (meaning it only uses $d+1$ vertices) we are done. Otherwise we can pick a vertex $v$ such that there exists a facet $F$ with $v \in F$ and another facet $G$ with $v \notin G$. We then have that $\delc(v)$ is a simplex (generated by the single facet $G$) and hence is vertex decomposable. The link $\lkc(v)$ is a pure $(d-1)$-dimensional complex on at most $d+1$ vertices, and hence is vertex decomposable by the induction hypothesis.  To see that $v$ is a shedding vertex note that the single facet of $\delc(v)$ is $G$, which was a facet of $\Delta$ by assumption.  The claim follows.
\end{proof}

\begin{thm}\label{thm:shellableVD}
Suppose $\Delta$ is a pure shellable $d$-dimensional simplicial complex on $d+3$ vertices.  Then $\Delta$ is vertex decomposable.
\end{thm}

\begin{proof}
We prove the statement by induction on $d$. If $d=1$ the claim follows from Lemma \ref{lem:graphVD}. Suppose $d \geq 1$ and let $\Delta$ be a $d$-dimensional complex with shelling order $F_1, F_2, \dots, F_j$.  Suppose $\Delta$ has vertex set $V$ and let $G$ be the simple graph on vertex set $V$ with \emph{non}edges $\{V \backslash F_i: F_i \in \Delta\}$, corresponding to the complements of facets. 

From Lemma \ref{lem:exposed} we have that $G$ is a chordal graph and hence admits a simplicial vertex $v$, so that the graph induced on its neighborhood $N_G(v)$ is a complete graph.  We first claim that $v$ is a shedding vertex of $\Delta$.  For a contradiction suppose $F$ is a facet of $\delc(v)$ such that $F \cup \{v\}$ is a facet of $\Delta$.  This implies that $\{i,j\} = V \backslash (F \cup \{v\})$ is missing as an edge in the complement graph $G$.  Since $v$ is simplicial this implies that $i$ or $j$ is not adjacent to $v$ in $G$.  Without loss of generality suppose $\{i,v\}$ is missing as an edge in $G$. This implies that $F^\prime = V \backslash \{i,v\}$ is a facet of $\Delta$. But $F \subsetneq F^\prime$ and $v \notin F^\prime$, a contradiction to the fact that $F$ is a facet of $\delc(v)$. 

Next note that $\delc(v)$ is a pure $d$-dimensional complex (since $v$ is a shedding vertex) on $d+2$ vertices. Hence by Lemma \ref{lem:smallVD} we have that $\delc(v)$ is vertex decomposable. Finally we have from Lemma \ref{lem:linkshell} that the link $\lkc(v)$ is a  shellable $(d-1)$-dimensional complex on at most $d+2$ vertices. By induction on $d$ we then have that $\lkc(v)$ is vertex decomposable.  The result follows.
\end{proof}

Hence for a $d$-dimensional simplicial complex on at most $d+3$ vertices the concepts of shellable, extendably shellable, shelling completable, and vertex decomposable are all equivalent.  We remark that Theorem \ref{thm:shellableVD} is tight in the sense that there exist $2$-dimensional complexes on 6 vertices that are shellable but not vertex decomposable \cite{MorTak}.

\section{Completing $k$-decomposable complexes and further thoughts}\label{sec:further}

Recall that Simon's conjecture posits that all shellable complexes are shelling completable, and in this paper we have shown the conjecture holds for the particular class of vertex decomposable complexes.  As alluded to in Section \ref{sec:intro}, we can ask the same question for complexes that properly sit in between these two classes. We first recall the relevant definitions.

\begin{defn}
A pure $d$-dimensional simplicial complex $\Delta$ is said to be \defi{$k$-decomposable} if $\Delta$ is a simplex, or $\Delta$ contains a face $F$ such that
\begin{enumerate}
    \item $\dim(F) \leq k$
    \item both $\delc(F)$ and $\lkc(F)$ are $k$-decomposable, and
    \item $\delc(F)$ is pure (and the dimensions stays same as that of $\Delta$).
\end{enumerate}
A face $F$ which satisfies the third condition is called a \defi{shedding face} of $\Delta$.
\end{defn}

The notion of $k$-decomposable interpolates between the notion of vertex decomposable (which is equivalent to $0$-decomposable in this language) and shellable (which can be seen to coincide with $d$-decomposable).

\begin{example}
Let $\Delta$ be the $2$-dimensional complex with facets
\[\{123,124,125,134,136,245,256,346,356,456\},\] (Example V6F10-6 from \cite{MorTak}). In \cite{MorTak} it is shown that $\Delta$ is not vertex decomposable, but one can check that it is $1$-decomposable using $15$ as a shedding face. 
\end{example}

The results in this paper imply that a $0$-decomposable complex is shelling completable, and Simon's conjecture posits that a $k$-decomposable complex is shelling completable.  As far we know it is an open question whether a $1$-decomposable complex is shelling completable.

To establish that a $1$-decomposable complex is shelling completable, one could try to generalize the argument given in the proof of Theorem \ref{thm:vdextend}, where the case of $0$-decomposable complexes was considered.  Recall that our strategy was to use induction on the number of vertices (and dimension), with the base case given by complexes that are \emph{full} (that is, are full skeleta of a simplex on their vertex set). It is not clear what the analogue for these complexes would be, even in the $1$-decomposable case. We leave this as a future project.

\section*{Acknowledgments}
 Much of this work was conducted under NSF-REU grant DMS-1757233 during the Summer 2020 Mathematics REU at Texas State University. The authors gratefully acknowledge the financial support of NSF and thank Texas State for providing a great working environment. We also thank Bruno Benedetti, Dylan Douthitt,  Jos\'e Samper, Jay Schweig, and Michelle Wachs for useful discussions and references.  We are also grateful to the anonymous referees for helpful corrections and suggestions.

\end{document}